\documentclass[a4paper,12pt,reqno]{amsart}  

\usepackage{amsmath}
\usepackage[all]{xy}
\usepackage{tikz-cd}
\usepackage{amssymb}
\usepackage{amsthm}
\usepackage{amsfonts}
\usepackage{mathtools}
\usepackage{comment}
\usepackage{mathrsfs}
\usepackage{pifont}
\usepackage{cite}
\usepackage{enumerate}
\usepackage{here}
\usepackage{autobreak}
\usepackage{wrapfig}
\usepackage{graphicx}
\usepackage{lscape}
\usepackage[colorlinks]{hyperref}  
\usepackage{xcolor}
\hypersetup{
	bookmarksnumbered=true,
    colorlinks=true,
    citecolor=blue,
    linkcolor=purple,
    urlcolor=orange,
}

\usepackage{pgfplots}
\pgfplotsset{compat=1.18}

\usepackage[capitalize, nameinlink]{cleveref}

\newcommand\Tstrut{\rule{0pt}{3ex}}         
\newcommand\Bstrut{\rule[-2ex]{0pt}{0pt}}   
\newcommand\BBstrut{\rule[-4ex]{0pt}{0pt}}   

\everymath{\displaystyle}
\numberwithin{equation}{section}

\theoremstyle{plain}

\newtheorem{thm}{Theorem}[section]

\newtheorem{lem}[thm]{Lemma}
\newtheorem{prop}[thm]{Proposition}

\newtheorem{introthm}{Theorem}[section]

\theoremstyle{definition}

\newtheorem{dfn}[thm]{Definition}

\newtheorem{eg}[thm]{Example}
\newtheorem*{Ack}{Acknowledgement}
\newtheorem*{NoCon}{Notation and Conventions}
\newtheorem*{Out}{Outline of this paper}

\theoremstyle{remark}
\newtheorem*{rem}{Remark}

\DeclareMathOperator{\Ext}{Ext}
\DeclareMathOperator{\ext}{ext}
\DeclareMathOperator{\Pic}{Pic}
\DeclareMathOperator{\Aut}{Aut}
\DeclareMathOperator{\Auteq}{Auteq}
\DeclareMathOperator{\eqAuteq}{Auteq_{\mathrm{eq}}}

\DeclareMathOperator{\NS}{NS}
\DeclareMathOperator{\ord}{ord}

\DeclareMathOperator{\ch}{ch}

\DeclareMathOperator{\rk}{rk}

\DeclareMathOperator{\id}{id}

\DeclareMathOperator{\pt}{pt}

\DeclareMathOperator{\SL}{SL}

\DeclareMathOperator{\Coh}{\mathbf{Coh}}

\DeclareMathOperator{\Image}{Im}

\DeclareMathOperator{\rad}{rad}

\DeclareMathOperator{\Num}{Num}
\DeclareMathOperator{\For}{For}
\DeclareMathOperator{\Inf}{Inf}

\newcommand\Hom{\mathop{\mathrm{Hom}}\nolimits}
\newcommand\RHom{\mathop{\mathbf{R}\mathrm{Hom}}\nolimits}
\newcommand{\PP}[1]{\mathbb{P}^{#1}}

\newcommand{\lderived}{\mathbf{L}}
\newcommand{\rderived}{\mathbf{R}}

\newcommand\dual{\raise0.9ex\hbox{$\scriptscriptstyle\vee$}}

\newcommand{\ords}{\ord(K_S)}
\newcommand{\htop}{h_{\mathrm{top}}}
\newcommand{\hcat}{h_{\mathrm{cat}}}
\newcommand{\numg}[1]{{#1}^{N}}

\newcommand{\xmark}{\ding{55}}

\newcommand{\CB}{\mathbb{C}}

\newcommand{\HB}{\mathbb{H}}

\newcommand{\RB}{\mathbb{R}}

\newcommand{\QB}{\mathbb{Q}}
\newcommand{\ZZ}{\mathbb{Z}}

\newcommand{\LC}{\mathcal{L}}

\newcommand{\PC}{\mathcal{P}}

\newcommand{\OO}{\mathscr{O}}

\author{Tomoki Yoshida}
\address[T.Yoshida]{Department~of~Mathematics, School~of~Science~and~Engineering, Waseda~University, Ohkubo~3-4-1, Shinjuku, Tokyo~169-8555, Japan}
\email{\href{mailto:tomoki_y@asagi.waseda.jp}{tomoki\_y@asagi.waseda.jp}}

\title[Categorical Entropy of Bielliptic and Enriques Surfaces]{A Note on Categorical Entropy of Bielliptic Surfaces and Enriques Surfaces}
\date{March 10, 2026}

\keywords{Derived category, Categorical Entropy, Bielliptic Surface, Enriques Surface}
\subjclass[2020]{14F08 (primary), 14J27, 14L30 (secondary).}

\begin{document}
\begin{abstract}
In this note, we show that there exists an autoequivalence of positive categorical entropy on the derived category of bielliptic surfaces.
This gives the first example of a surface admitting positive categorical entropy in the absence of both positive topological entropy and any spherical objects. 
Moreover, we prove a Gromov–Yomdin type equality for the categorical entropy of autoequivalences on bielliptic surfaces and give a counterexample to this equality on Enriques surfaces.
\end{abstract}
\maketitle
\setcounter{tocdepth}{2} 
\tableofcontents
\setcounter{section}{-1}
\section{Introduction}\label{section: introduction}

\subsection{Entropies on Algebraic Surfaces}\label{subsection: entropies on algebraic surfaces}
Let $X$ be a smooth compact K\"ahler surface and $f\in\Aut(X)$ an automorphism of $X$.
Cantat \cite{cantat_1999_dynamique_des_automorphismes_des_surfaces_projectives_complexes} showed that if $X$ admits an automorphism $f$ of positive topological entropy, then $X$ must be either a rational surface, a K3, an Enriques, or a complex torus.

From the point of view of the derived category, the group of autoequivalences $\Auteq(D^b(X))$ may have a much richer structure than $\Aut(X)$.
Let $A(X)$ be the subgroup of $\Auteq(D^b(X))$ that consists of standard autoequivalences, more precisely, 
\[
A(X)\coloneqq \Pic(X)\rtimes\Aut(X)\times\ZZ[1], 
\]
where $\LC\in\Pic(X)$ and $f\in\Aut(X)$ act as $-\otimes\LC$ and $\lderived f^{\ast}$, respectively.

Ouchi discovered that any projective K3 surface admits an autoequivalence of positive categorical entropy in \cite{ouchi_2018_automorphisms_of_positive_entropy_on_some_hyperkahler_manifolds_via_derived_automorphisms_of_k3_surfaces}.
Mattei \cite{mattei_2021_categorical_vs_topological_entropy_of_autoequivalences_of_surfaces} has shown that every variety with a $(-2)$-curve also has that property.
In these works, the spherical twists, which are non-standard autoequivalences, play an important role in constructing suitable autoequivalences.
Note that these results imply that a variety may have an autoequivalence with positive categorical entropy even if it has no automorphism of positive topological entropy. For instance, the Hirzebruch surface $\mathbb{F}_2$ has no automorphism of positive topological entropy but has a spherical object.
Thus, our main interest is the case of $A(X)\neq\Auteq(D^b(X))$.
Note that Bondal and Orlov
\cite{bondal_orlov_2001_reconstruction_of_a_variety_from_the_derived_category_and_groups_of_autoequivalences} showed that $A(X)=\Auteq(X)$ if the canonical sheaf $\omega_X$ is (anti-)ample. 

A \emph{Bielliptic surface} $S$ is one of the algebraic surfaces of $\kappa(S)=0$.
This class does not belong to Cantat's list and has no spherical object (see \cref{proposition: no spherical object on bielliptic surface}).
However, $S$ admits non-standard autoequivalences known as \emph{relative Fourier-Mukai transform}.
The relative Fourier-Mukai transform is an essential tool for investigating $\Auteq(S)$.
Indeed, Potter and Tochitani claim that $\Auteq(S)$ is generated by the standard autoequivalences and relative Fourier-Mukai transforms (\cite{preprint_rory_2017_derived_autoequivalences_of_bielliptic_surfaces,preprint_yuki_2026_autoequivalences_of_derived_category_of_bielliptic_surfaces}).

The first observation of this paper is the following:

    \begin{introthm}[See \cref{theorem: example of positive categorical entropy}]\label{theorem in intro: example of positive categorical entropy}
        Let $S$ be a bielliptic surface. 
        Then there exists an autoequivalence $\Phi\in\Auteq(D^b(S))$ such that 
        \[
        \hcat(\Phi) > 0.
        \]
    \end{introthm}

\subsection{Gromov-Yomdin type equality}\label{subsection: Gromov-Yomdin type equality}
The classical Gromov-Yomdin equality asserts the following: 
\begin{introthm}[Gromov-Yomdin equality, \cite{gromov_1987_entropy_homology_and_semialgebraic_geometry,gromov_2003_on_the_entropy_of_holomorphic_maps,yomdin_1987_volume_growth_and_entropy}]\label{theorem in intro: Gromov-Yomdin equality}
    Let $f\in \Aut(X)$. Then, 
    \[
    \htop(f) = \log\rho(f^{\ast}|_{\oplus H^{p,p}(X)}), 
    \]
    where the $\rho$ is the spectral radius.
\end{introthm}

Kikuta and Takahashi \cite{kikuta_takahashi_2019_on_the_categorical_entropy_and_the_topological_entropy} conjectured that an analogous formula for categorical entropy also holds, i.e. 
    \[
    \hcat(\Phi) = \log\rho(\numg{\Phi}), 
    \]
    where $\numg{\Phi}$ is the induced action on the numerical Grothendieck group $N(X)$ (see \cref{section: preliminaries cat entropy}).
They verified it for the cases where either $X$ is a curve, an orbifold projective line, or $\Phi$ is a standard autoequivalence \cite{kikuta_2017_on_entropy_for_autoequivalences_of_the_derived_category_of_curves,kikuta_takahashi_2019_on_the_categorical_entropy_and_the_topological_entropy,kikuta_shiraishi_takahashi_2020_a_note_on_entropy_of_autoequivalences_lower_bound_and_the_case_of_orbifold_projective_lines}.
Also, Yoshioka \cite{yoshioka_2020_categorical_entropy_for_fouriermukai_transforms_on_generic_abelian_surfaces} showed that it is true for abelian surfaces. 

After that, counterexamples to this conjecture have been found for the even-dimensional Calabi-Yau hypersurfaces \cite{fan_2018_entropy_of_an_autoequivalence_on_calabiyau_manifolds}, any projective K3 surface \cite{ouchi_2020_on_entropy_of_spherical_twists}, and surfaces with $(-2)$-curve \cite{mattei_2021_categorical_vs_topological_entropy_of_autoequivalences_of_surfaces}.

The second observation of this note is the confirmation of this equality for bielliptic surfaces by passing to the canonical cover of $S$ and Yoshioka's result:

\begin{introthm}[See \cref{theorem: Gromov-Yomdin equality for bielliptic surface}]\label{Theorem in intro: Gromov-Yomdin for bielliptic surface}
        Let $S$ be a bielliptic surface and $\Phi\in\Auteq D^b(S)$.
    Then, the equality
    \[
        \hcat(\Phi) = \log(\rho(\numg{\Phi}))
    \]
     holds. 
\end{introthm}

These discussions can also apply to the case of the Enriques surfaces.
With the counterexample of Gromov-Yomdin type equality \cite{ouchi_2020_on_entropy_of_spherical_twists}, we have the following result:
\begin{introthm}[See \cref{theorem: counter eg Enriques}]\label{theorem in intro: counter eg in enriques}
    Let $S$ be an Enriques surface. Then, there exists an autoequivalence $\Phi\in\Auteq D^b(S)$ such that
    \[
        \hcat(\Phi) > \log\rho(\numg{\Phi}).
    \]
\end{introthm}

In summary, our results can be settled into the following table:
\begin{table}[ht]
    \centering
    \begin{tabular}{c|c|c|c|c}
    Surface type & Positive $\htop$ & spherical object & Positive $\hcat$ & \begin{tabular}{c}Gromov-Yomdin \\ for $\hcat$ \end{tabular}\\ \hline
    K3 & \checkmark & \checkmark &\checkmark via spherical twist & \xmark \Tstrut \Bstrut \\
    Abelian& \checkmark & \xmark &\checkmark via $\htop>0$ & \checkmark \Tstrut \Bstrut \\
    \begin{tabular}{c}Surface with \\ $(-2)$-curve\end{tabular}
    & generally \xmark & \checkmark &\checkmark via spherical twist & \xmark\Tstrut \BBstrut \\
    \textbf{Bielliptic}& \xmark & \xmark & 
    \begin{tabular}{c}previously no example \\ $\rightsquigarrow$ \ \cref{theorem in intro: example of positive categorical entropy}\end{tabular}
    & 
    \begin{tabular}{c}
    \checkmark\\
    \cref{Theorem in intro: Gromov-Yomdin for bielliptic surface}
    \end{tabular}
    \Tstrut \BBstrut \\
    \textbf{Enriques} & \checkmark & generically \xmark & \checkmark &
    \begin{tabular}{c}
        \xmark\\
        \cref{theorem in intro: counter eg in enriques}
    \end{tabular}
    \end{tabular}
    \vspace{5mm}
    \caption{Summary of results}
    \label{table: summary of results}
\end{table}

\begin{rem}
    \cite[Proposition 3.17]{macri_mehrotra_stellari_2009_inducing_stability_conditions} shows that the generic Enriques surface admits no spherical object. 
\end{rem}

\begin{Out}
In \cref{section: preliminaries cat entropy}, we review the definition and fundamental properties of categorical entropy.
Then, \cref{section: bielliptic surfaces} fixes notations for bielliptic surfaces and 
recall results on relative Fourier-Mukai transforms, which will be used especially in the following section. 

In \cref{section: Example of Autoequivalence with Positive Entropy}, we prove \cref{theorem in intro: example of positive categorical entropy}.
Finally, \cref{section: proof of Gromov-Yomdin type equality} is devoted to the proof of \cref{Theorem in intro: Gromov-Yomdin for bielliptic surface} and \cref{theorem in intro: counter eg in enriques}. 
\end{Out}

\begin{NoCon}
The conventions and notations used in this paper are listed below:
\begin{itemize}
    \item all varieties are smooth and projective defined over $\CB$, 
    \item the bounded derived category of coherent sheaves $D^b(\Coh(X))$ is denoted by $D^b(X)$ for short,
    \item $h^i(E) \coloneqq \dim H^i(E)$ for $E\in D^b(X)$,
    \item $\ext^i(E, F)\coloneqq \dim\Ext^{i}(E, F)  = \dim\Hom(E, F[i]) =\hom^i(E, F)$ for $E, F\in D^b(X)$, and
    \item all functors are derived unless otherwise stated. However, $\lderived$ and $\rderived$ are occasionally used to emphasize that the functor is actually derived.
\end{itemize}
\end{NoCon}

\begin{Ack}
    I am grateful to Prof. Yasunari Nagai for regular discussions and comments.
    I would also like to thank Prof. Genki Ouchi for reading and comments on the draft version of this paper, especially in the simplification of the proof of \cref{Theorem in intro: Gromov-Yomdin for bielliptic surface},  commenting on \cref{subsection: case of enriques surface}, and introducing the paper \cite{ploog_2007_equivariant_autoequivalences_for_finite_group_actions} to me and Dr. Yuki Tochitani for informing me of the content of her preprint \cite{preprint_yuki_2026_autoequivalences_of_derived_category_of_bielliptic_surfaces}.
    I am also grateful to the anonymous referee for carefully reading the manuscript and for pointing out numerous typographical errors.
    During the preparation of this article, the author was supported by JST SPRING, Grant Number JPMJSP2128.
\end{Ack}
\section{Preliminaries on Categorical Entropies}\label{section: preliminaries cat entropy}

In this section, we review the definitions and properties of derived equivalences and their entropies
assuming basic familiarity with derived categories of coherent sheaves. 
For the basics of derived categories and autoequivalences, see \cite{book_huybrechts_2006_fouriermukai_transforms_in_algebraic_geometry}.

In this paper, we treat only the derived category of coherent sheaves on a smooth projective variety as a triangulated category. 
Thus, we introduce the notion of categorical entropy in a simple form, which is suitable for our setting. See \cite{dimitrov_haiden_katzarkov_kontsevich_2014_dynamical_systems_and_categories} for a more precise definition.

\begin{thm}[\cite{dimitrov_haiden_katzarkov_kontsevich_2014_dynamical_systems_and_categories}]\label{definition: categorical entropy}
    Let $X$ be a smooth projective variety, $\Phi: D^b(X)\to D^b(X)$ an exact functor, $t\in\RB$, and $G, G'$ be classical generators of $D^b(X)$.
    The entropy of $\Phi$ is defined by
    \[
        h_t(\Phi) \coloneqq \lim_{n\to \infty} \frac{1}{n}\log \big(\sum_{m\in\ZZ} \ext^{m}(G, \Phi^{n}G')\big)e^{-mt}.
    \]
    Especially, $\hcat(\Phi)\coloneqq h_{0}(\Phi)$ is called categorical entropy.
    Then, the limit exists in $[-\infty, \infty)$ and is independent of the choices of classical generators. Moreover, $\hcat(\Phi)\ge0$.
\end{thm}

\begin{rem}\label{remark: existence of generator}
    The generator always exists for our setting.
    Indeed, let $\OO(1)$ be a very ample line bundle on a smooth projective variety $X$.
    Orlov \cite{orlov_2009_remarks_on_generators_and_dimensions_of_triangulated_categories} have shown that 
    \[
    \bigoplus_{i=0}^{\dim X}\OO(i+k)
    \]
    is a generator of $D^b(X)$ for any $k\in\ZZ$.
\end{rem}

Next, we recall that an autoequivalence induces an action on the numerical Grothendieck group.
For more details, see \cite[Section 5.2]{book_huybrechts_2006_fouriermukai_transforms_in_algebraic_geometry}.

Let $X$ be a smooth projective variety over $\CB$ and 
$\rad\chi\coloneqq \{x\in K(X) \mid \chi(x, y) = \chi(y, x) =0 \text{ for any } y\in K(X)\}$.
The \emph{numerical Grothendieck group} of $X$ is defined and denoted by $N(X)\coloneqq K(X)/\rad\chi$.
For an exact functor $\Phi: D^b(X) \to D^b(X)$, there is the induced linear map on the numerical Grothendieck group by $\numg{\Phi}: N(X)_{\CB} \to N(X)_{\CB}$.
Moreover, if $\Phi$ is an equivalence, $\numg{\Phi}$ is also an isomorphism.

There are some partial results of the lower bound of the categorical entropy by the spectral radius, which is the categorical version of the Yomdin inequality
\cite{ikeda_2021_mass_growth_of_objects_and_categorical_entropy,kikuta_shiraishi_takahashi_2020_a_note_on_entropy_of_autoequivalences_lower_bound_and_the_case_of_orbifold_projective_lines}).
In this paper, we shall directly use the lower bound results in the following form:

\begin{thm}[{\cite[Proposition 4.7.]{ikeda_2021_mass_growth_of_objects_and_categorical_entropy}}]\label{theorem: lower bound - Yomdin inequality for surfaces}
    Let $X$ be a smooth projective surface over $\CB$. and $\Phi: D^b(X)\to D^b(X)$ be a Fourier-Mukai type endofuncor. Then, the following holds:
    \[
    \log\rho(\numg{\Phi}) \le \hcat(\Phi).
    \]
\end{thm}

On the other hand, Gromov-Yomdin type equality has been confirmed for some cases, for curves, standard autoequivalence, and abelian surfaces, which is important for this note.

\begin{thm}[{\cite[Proposition 2.11.]{yoshioka_2020_categorical_entropy_for_fouriermukai_transforms_on_generic_abelian_surfaces}}]\label{theorem: Gromov-Yomdin for an abelian surface}
    Let $X$ be an abelian surface and $\Phi: D^b(X) \to D^b(X)$ an equivalence.
    Then, 
    \[
    \hcat(\Phi) \le \log \rho(\numg{\Phi}).
    \]
\end{thm}
\begin{rem}
Note that combining the lower bound \cref{theorem: lower bound - Yomdin inequality for surfaces}, 
we have Gromov-Yomdin type equality for any abelian surface (\cite[Proposition 2.10.]{yoshioka_2020_categorical_entropy_for_fouriermukai_transforms_on_generic_abelian_surfaces}).
\end{rem}

\section{Basics on the Bielliptic Surfaces}\label{section: bielliptic surfaces}

In this \cref{section: bielliptic surfaces}, we introduce some basic facts and fix the notations about the bielliptic surfaces. See \cite[Section 2.]{nuer_2025_stable_sheaves_on_bielliptic_surfaces_from_the_classical_to_the_modern} for more details.

\begin{dfn}
    Let $S$ be a smooth projective surface over $\CB$.
    $S$ is called \emph{biellipic surface} if $\kappa(S)=0$, 
    $q(S)=h^{1,0}(S) =1$, and $p_g(S) = h^{2,0}(S) = 0$.
\end{dfn}
It is known that for a bielliptic surface $S$, there exist elliptic curves $E, F$ and a finite group $G$ such that $S\cong(E\times F)/G$ and such that $G$ acts on $E$ by translation and on $F$ in a way so that $F/G\cong\PP{1}$.

Denote $F\cong\CB/(\ZZ\oplus\ZZ\tau)$ ($\tau\in\HB$), and
\[
\ords\coloneqq \min\{m>0 \mid mK_S = 0\}.
\]
It is also known that bielliptic surfaces can be classified into 7 types (\cref{table: Classification of Bielliptic Surfaces}). See \cite[List VI.20]{book_beauville_1996_complex_algebraic_surfaces} for instance.

\begin{table}[ht]
    \centering
    \begin{tabular}{c|c|c|c|c}
        Type & $\tau$ & $G_S$ & Action on $F$ & $\ords$ \\ \hline
        1 & \text{any} & $\ZZ/2\ZZ$ & $x\mapsto -x$ & 2 \Tstrut \Bstrut \\
        2 & \text{any} & $\ZZ/2\ZZ\times\ZZ/2\ZZ$ & 
        \begin{tabular}{c}
            $x\mapsto -x$, \\
            $x\mapsto x+\epsilon ~~~~ (\epsilon\in\ZZ/2\ZZ)$
        \end{tabular}
        & 2 \Tstrut \Bstrut \\
        3 & $i$ & $\ZZ/4\ZZ$ & $x\mapsto ix$ & 4 \Tstrut \Bstrut \\
        4 & $i$ & $\ZZ/4\ZZ\times\ZZ/2\ZZ$ & 
        \begin{tabular}{c}
            $x\mapsto ix$, \\
            $x\mapsto x+\bigg(\frac{1+i}{2}\bigg)$
        \end{tabular}
        & 4 \Tstrut \Bstrut \\
        5 & $\omega$ & $\ZZ/3\ZZ$ & $x\mapsto \omega x$ & 3 \Tstrut \Bstrut \\
        6 & $\omega$ & $\ZZ/3\ZZ\times\ZZ/3\ZZ$ & 
        \begin{tabular}{c}
            $x\mapsto \omega x$, \\
            $x\mapsto x+\bigg(\frac{1-\omega}{3}\bigg)$
        \end{tabular}
        & 3 \Tstrut \Bstrut\\
        7 & $\omega$ & $\ZZ/6\ZZ$ & $x\mapsto -\omega x$ & 6 \Tstrut \Bstrut 
    \end{tabular}
    \vspace{5mm}
    \caption{Classification of Bielliptic Surfaces}
    \label{table: Classification of Bielliptic Surfaces}
\end{table}

Let $n\coloneqq \ords$, and $k\coloneqq |G|/n$. 
Additionally, we denote $A\coloneqq \frac{1}{n} E$ and $B \coloneqq \frac{1}{k} F$.
Then, we have 
\[
\Num(S) = \ZZ[A]\oplus\ZZ[B].
\]
and their intersection form is $A^2 = B^2 =0, A.B=1$.

Note that, in our case,the image of the map
\[
\ch: K(X)_{\QB} \to H^{\ast}(X, \QB)
\]
is equal to $N(X)$.
For a bielliptic surface $S$, 
$N(S) = H^0(S, \ZZ)\oplus \Num(S) \oplus H^4(S, \ZZ)\cong\ZZ\oplus A\ZZ \oplus B\ZZ\oplus\ZZ$
(see also \cite[Proposition 2.5.]{preprint_rory_2017_derived_autoequivalences_of_bielliptic_surfaces}).
We fix and denote the basis of $N(S)$ as
\begin{equation}
\begin{split}
[S] = \begin{pmatrix}
    1&0&0&0
\end{pmatrix}, ~~~
&A = \begin{pmatrix}
    0&1&0&0
\end{pmatrix}, \\
B = \begin{pmatrix}
    0&0&1&0
\end{pmatrix}, \text{ and }
&[\pt] = \begin{pmatrix}
    0&0&0&1
\end{pmatrix}. 
\end{split}
\label{equation: basis of numerical Grothendieck group}
\end{equation}

If there is a spherical twist (associated to a spherical object), the composition of the spherical twists possibly has a positive entropy \cite{barbacovi_kim_2023_entropy_of_the_composition_of_two_spherical_twists}.
However, it can be deduced that the bielliptic surface can admit no spherical object.

\begin{prop}[{\cite[Remark 4.8.]{sosna_2013_fouriermukai_partners_of_canonical_covers_of_bielliptic_and_enriques_surfaces}}]\label{proposition: no spherical object on bielliptic surface}
    Let $S$ be a bielliptic surface.
    Then, there is no spherical object in $D^b(S)$. 
\end{prop}

\subsection{Relative Fourier-Mukai Transforms}\label{subsection: relative FMT}
Let $\pi: X\to C$ be a relatively minimal elliptic surface. 
Denote 
\begin{itemize}
\item $f$: class of smooth fiber of $\pi$, 
\item $\lambda \coloneqq \min\{f.D>0 \mid D\in\Num(X)\}$, and
\item $J(a,b)$: moduli space of stable and pure dimension $1$ supported on a smooth fibre of $\pi$ sheaves of class $(a,b)$, where $a>0$, $b\in\ZZ$, and $\gcd(a\lambda,b)=1$.
\end{itemize}
It is known that the moduli space $J(a,b)$ is derived equivalent to $X$ by a functor called \emph{relative Fourier-Mukai transform} \cite{bridgeland_1998_fouriermukai_transforms_for_elliptic_surfaces,bridgeland_maciocia_2002_fouriermukai_transforms_for_k3_and_elliptic_fibrations}.

\begin{thm}[{\cite[Theorem 5.3.]{bridgeland_1998_fouriermukai_transforms_for_elliptic_surfaces}}]
\label{theorem: relative FMT}
    Let $\pi: X \to C$ be an elliptic surface and 
    \[
    \begin{pmatrix}
        c & a\\
        d & b
    \end{pmatrix}
    \in \SL(2, \ZZ)
    \]
    such that $\lambda$ divides $d$ and $a>0$.
    Then, there exists a sheaf $\PC$ on $X\times J(a,b)$, which is flat, strongly simple over both factors, and 
    $\ch\PC_y = (0. af, b)$ on $X$ and 
    $\ch\PC_x = (0. af, c)$ on $J(a,b)$ 
    for any $(x,y)\in X\times J(a,b)$.

    Furthermore, for any such sheaf $\PC$, the Fourier-Mukai functor $\Phi_{\PC}: D^b(J(a,b)) \to D^b(X)$
    is an equivalence and for any $E\in D^b(J(a,b))$,
    \begin{equation}
        \begin{pmatrix}
            \rk(\Phi_{\PC}(E))\\
            c_1(\Phi_{\PC}(E)).f
        \end{pmatrix}
        = 
        \begin{pmatrix}
            c & a\\
            d & b
        \end{pmatrix}
        \begin{pmatrix}
            \rk(E)\\
            c_1(E).f
        \end{pmatrix}.
        \notag
    \end{equation}
\end{thm}

Especially, as $S$ has no non-isomorphic Fourier-Mukai partners, the relative Fourier-Mukai transformation induces an element $\Phi\in\Auteq(D^b(S))$.
The relative Fourier-Mukai transform is important in studying derived equivalences of a bielliptic surface.
\begin{prop}[{\cite[Proposition 4.1.]{preprint_rory_2017_derived_autoequivalences_of_bielliptic_surfaces}}]\label{propositon: relative FMT is not standard}
    Let $S$ be a bielliptic surface and $p: S\to E/G$ and $q: S\to F/G$ its fibrations.
    Then, for each fibration, the induced relative Fourier-Mukai transform is not a standard autoequivalence.
\end{prop}
\begin{thm}[{\cite[Theorem 1.2.]{preprint_rory_2017_derived_autoequivalences_of_bielliptic_surfaces}}, \cite{preprint_yuki_2026_autoequivalences_of_derived_category_of_bielliptic_surfaces}]
    Let $S$ be a bielliptic surface. Then, $\Auteq D^b(S)$ is generated by standard autoequivalences and relative Fourier-Mukai transforms along the two elliptic fibrations.
\end{thm}

\begin{eg}[{\cite[Example 4.2.]{preprint_rory_2017_derived_autoequivalences_of_bielliptic_surfaces}
}]\label{example: Potter eg 4.2.}
 Let 
    \[
    P = \begin{pmatrix}
        1 & 1\\
        0 & 1
    \end{pmatrix}
    \]
    and $\Phi_P$ be a relative Fourier-Mukai transform along $p: S\to E/G$ which is associated to $P$ 
    (and tensoring a suitable line bundle).
    Then, $\numg{\Phi}_P$ acts as
    \begin{align*}
            \numg{\Phi}_P
            \begin{pmatrix}
                1&0&0&0
            \end{pmatrix} &= 
            \begin{pmatrix}
                1&0&0&0
            \end{pmatrix}\\
            
            \numg{\Phi}_P
            \begin{pmatrix}
                0&1&0&0
            \end{pmatrix} &= 
            \begin{pmatrix}
                n&1&0&0
            \end{pmatrix}\\
            
            \numg{\Phi}_P
            \begin{pmatrix}
                0&0&1&0
            \end{pmatrix} &= 
            \begin{pmatrix}
                0&0&1&0
            \end{pmatrix}\\
            
            \numg{\Phi}_P
            \begin{pmatrix}
                0&0&0&1
            \end{pmatrix} &= 
            \begin{pmatrix}
                0&0&k&1
            \end{pmatrix}
    \end{align*}
\end{eg}

\section{Example of Autoequivalence with Positive Entropy}\label{section: Example of Autoequivalence with Positive Entropy}
    \begin{thm}\label{theorem: example of positive categorical entropy}
        Let $S$ be a bielliptic surface. 
        Then there exists an autoequivalence $\Phi\in\Auteq(D^b(S))$ such that 
        \[
        \hcat(\Phi) > 0.
        \]
    \end{thm}

    \begin{rem}
        Standard autoequivalences act via upper-triangular matrices under the identification \eqref{equation: basis of numerical Grothendieck group}.
        
        By contrast, non-standard autoequivalences send $\mathcal{O}_x$ to objects with nonzero rank or nontrivial first Chern class, which produces nonzero entries in the lower-triangular part of the action matrix on the numerical Grothendieck group.
        Indeed, an autoequivalence is standard if and only if it sends a skyscraper sheaf $\mathcal{O}_x$ of a point to another skyscraper sheaf up to a shift.
        
        Consequently, although each action taken individually need not produce positive categorical entropy, their composition can acquire eigenvalues with modulus greater than $1$.
        In the following proof, we explicitly construct such a composition by combining a relative Fourier–Mukai transform with a standard autoequivalence.
    \end{rem}
    \begin{proof}
    Let $H\coloneqq E+F = nA+kB$, 
    $\Phi_{\mathcal{P}}$ be the relative Fourier-Mukai transform in \cref{example: Potter eg 4.2.}, and 
    $\Phi_m\coloneqq \Phi_{\mathcal{P}}\circ (-\otimes \OO(-mH))$ for an $m\in\ZZ_{>0}$.
        Then, we will see $\hcat(\Phi_m) > 0.$
        As for an object $E$ of rank $r$ 
        \[
        \ch(E\otimes\OO(-mH)) = r + c_1(E)-mr(nA+kB) + \ch_2 E - mc_1(E)(nA+kB) +\frac{r}{2}(-mH)^2, 
        \]
        $\numg{(-\otimes\OO(-mH))}$ acts on $N(S)$ as 
        \begin{align*}
            \numg{(-\otimes\OO(-mH))}
            \begin{pmatrix}
                1&0&0&0
            \end{pmatrix} &= 
            \begin{pmatrix}
                1&-mn&-mk&m^2nk
            \end{pmatrix},\\
            
            \numg{(-\otimes\OO(-mH))}
            \begin{pmatrix}
                0&1&0&0
            \end{pmatrix} &= 
            \begin{pmatrix}
                0&1&0&-mk
            \end{pmatrix},\\
            
            \numg{(-\otimes\OO(-mH))}
            \begin{pmatrix}
                0&0&1&0
            \end{pmatrix} &= 
            \begin{pmatrix}
                0&0&1&-mn
            \end{pmatrix},\text{ and}\\
            
            \numg{(-\otimes\OO(-mH))}
            \begin{pmatrix}
                0&0&0&1
            \end{pmatrix} &= 
            \begin{pmatrix}
                0&0&0&1
            \end{pmatrix}
        \end{align*}
        via the identification \eqref{equation: basis of numerical Grothendieck group}.
        Denote its representation matrix
        \[
        M_1 = \begin{pmatrix}
            1&-mn&-mk&m^2nk\\
            0&1&0&-mk\\
            0&0&1&-mn\\
            0&0&0&1
        \end{pmatrix}.
        \]
    Next functor is \cite[Example 4.2.]{preprint_rory_2017_derived_autoequivalences_of_bielliptic_surfaces}.
    $\numg{\Phi}_P$ acts as
    \begin{align*}
            \numg{\Phi}_P
            \begin{pmatrix}
                1&0&0&0
            \end{pmatrix} &= 
            \begin{pmatrix}
                1&0&0&0
            \end{pmatrix}, \\
            
            \numg{\Phi}_P
            \begin{pmatrix}
                0&1&0&0
            \end{pmatrix} &= 
            \begin{pmatrix}
                n&1&0&0
            \end{pmatrix},\\
            
            \numg{\Phi}_P
            \begin{pmatrix}
                0&0&1&0
            \end{pmatrix} &= 
            \begin{pmatrix}
                0&0&1&0
            \end{pmatrix},\text{ and}\\
            
            \numg{\Phi}_P
            \begin{pmatrix}
                0&0&0&1
            \end{pmatrix} &= 
            \begin{pmatrix}
                0&0&k&1
            \end{pmatrix}.
    \end{align*}

    Denote its representation matrix
    \[
        M_2 = \begin{pmatrix}
            1&0&0&0\\
            n&1&0&0\\
            0&0&1&0\\
            0&0&k&1
        \end{pmatrix}
        \]
    Then,
    \[
    M_2M_1 = \begin{pmatrix}
            1&-mn&-mk&m^2kn\\
            n&1-mn^2&-mkn&-mk+m^2kn^2\\
            0&0&1&-mn\\
            0&0&k&1-mkn
    \end{pmatrix}.
    \]
    and its eigenvalues are listed in \cref{table: example of positive spectral radius}:
    \begin{table}[ht]
        \centering
        \begin{tabular}{c|c|c|c}
        Type & $m$ & eigenvalues of $M_2M_1$& $\rho$ \\ \hline
        $1$ & 2 & $-1, -3\pm2\sqrt{2}$ & $3+2\sqrt{2}$ \Tstrut \Bstrut \\ 
        $2$ & 2 & $-3\pm2\sqrt{2}$ & $3+2\sqrt{2}$ \Tstrut \Bstrut \\ 
        $3$ & 1 & $-1, -7\pm4\sqrt{3}$ & $7+4\sqrt{3}$ \Tstrut \Bstrut \\ 
        $4$ & 1 & $-7\pm4\sqrt{3}, -3\pm2\sqrt{2}$ & $7+4\sqrt{3}$ \Tstrut \Bstrut \\ 
        $5$ & 1 & $-\omega, \omega^2, \frac{1}{2}\big(-7\pm3\sqrt{5}\big)$ & $\frac{1}{2}\big(7+3\sqrt{5}\big)$ \Tstrut \Bstrut \\ 
        $6$ & 1 & $\frac{1}{2}\big(-7\pm3\sqrt{5}\big)$ & $\frac{1}{2}\big(7+3\sqrt{5}\big)$ \Tstrut \Bstrut \\ 
        $7$ & 1 & $-2\pm\sqrt{3}, -17\pm12\sqrt{2}$ & $17+12\sqrt{2}$ \Tstrut \Bstrut \\ 
        \end{tabular}
        \vspace{5mm}
        \caption{Spectral radii of $\numg{\Phi}_P\circ\numg{(-\otimes\OO(-mH))}$}
        \label{table: example of positive spectral radius}
    \end{table}
    
    In particular, for every bielliptic surface, there exists an autoequivalence with positive categorical entropy.
        \end{proof}

\section{Proof of Gromov-Yomdin Type Equality}\label{section: proof of Gromov-Yomdin type equality}

\subsection{Canonical Covers}\label{subsection: canonical cover}
In this \cref{subsection: canonical cover}, we review the relation between the canonical cover and the derived equivalence.
Let $S$ be a smooth surface of finite canonical order and 
\[
 \pi: X \to S
\]
the canonical cover of $S$.
Note that $\pi$ is an \'etale cover of degree $\ords$.
Moreover, $\pi: X \to X/C_S\cong S$ is the quotient morphism,where $C_S=\ZZ/\ords\ZZ$.
Then, naturally, 
\begin{equation}
    \lderived\pi^{\ast}: D^{b}(S) \stackrel{\cong}{\longrightarrow} D^b_{C_S}(X).
    \label{equation: pi pullback from S to D_G(X)}
\end{equation}

The canonical covering map induces the following inclusion 
\[
\pi^{\ast}:N(S)_{\RB}=\RB\oplus \NS(S)_{\RB} \oplus\RB \hookrightarrow \RB\oplus \NS(X)_{\RB} \oplus\RB =N(X)_{\RB}
\]
as $\pi_{\ast}\pi^{\ast}\alpha = \ords\alpha$ for any $\alpha\in N(S)$.

Let $K$ be the orthogonal complement of $\Image(\pi^{\ast})$ in $N(X)$ i.e. $N(X) = \pi^{\ast}N(S)\oplus K$ and $K= \ker(\pi_{\ast})$.
Denote $\ell\coloneqq \dim_{\RB} K_{\RB}$. 

Note that when $S$ is a bielliptic surface, $\ords= 2,3,4, $ or $6$, $X$ is an abelian surface, and $0\le \ell\le 2$.

The Theorem below is on 
{\cite[Section 7.3]{book_huybrechts_2006_fouriermukai_transforms_in_algebraic_geometry}},
originally in \cite{bridgeland_maciocia_2001_complex_surfaces_with_equivalent_derived_categories}.

\begin{thm}\label{theorem: equivalence induce eq of canonical cover}
Suppose $S$ and $T$ are smooth projective varieties with canonical sheaves of finite order.
Assume that $S$ is derived equivalent to $T$.
Then, any equivalence $\Phi: D^b(S) \to D^b(T)$ admits an equivariant lift $\Psi: D^b(X) \to D^b(Y)$, where $Y$ is the canonical cover of $T$.
\end{thm}

\begin{prop}\label{proposition: cat entropy of lift}
    With the settings in \cref{theorem: equivalence induce eq of canonical cover}, entropies of $\Phi$ and $\Psi$ coincide. 
    Especially, the equation between categorical entropies holds, i.e. 
    \[
        \hcat(\Phi) = \hcat(\Psi).
    \]
\end{prop}
\begin{proof}
    Let $\OO_S(1)$ be a very ample line bundle on $S$.
    As $\pi$ is finite, the pullback of an ample divisor on $S$ is ample again on $X$. 
    Thus, we may assume that $\pi^{\ast}\OO(1)$ is also very ample on $X$ if we take $c_1(\OO(1))$ to be sufficiently large.
    Let $F\coloneqq \bigoplus_{i=0}^{\dim S}\OO_S(i)$, which is a classical generator of $D^b(S)$.
    Then, $F_X\coloneqq \lderived\pi^{\ast} F$ and 
    \begin{align}
    \pi_{\ast}\pi^{\ast}F &\cong F\otimes\pi_{\ast}\OO_X \notag\\
    &\cong F\otimes(\oplus_{i=0}^{m-1} \omega_{S}^{i}) \notag
    \end{align}
    are classical generators of $D^b(X)$ and $D^b(S)$ respectively as $\pi_*\pi^*F$ contains $F$ as a direct summand.
    For any $n\in \ZZ_{>0}$, 
    \begin{align}
        \RHom(F, \Phi^n(\pi_{\ast}\pi^{\ast}(F))) 
        &\cong \RHom(F, \pi_{\ast}\Psi^n(\pi^{\ast}F)) \notag \\
        &\cong \RHom(F_X, \Psi^n(F_X)). \notag
    \end{align}
    Therefore, as the categorical entropy is independent of the choice of the generator, we have 
    \[
        h_t(\Phi) = h_t(\Psi).
    \]
    In particular, if we set $t=0$ we have $\hcat(\Phi) = \hcat(\Psi)$.
\end{proof}

\begin{lem}\label{lemma: eigenvalue of Psi on K}
    In the same settings in \cref{theorem: equivalence induce eq of canonical cover}, assume that $S=T$, $X(=Y)$ is an abelian surface or a K3 surface, and $l>0$.
    Let $\numg{\Psi}: N(X)_{\RB} \to N(Y)_{\RB}$ be the action defined by the lift $\Psi$.
    Then, the spectral radius of the linear map $\numg{\Psi}|_{K}$ is $1$. 
\end{lem}

\begin{proof}
    Let $[c_1(\LC_1)], \ldots, [c_1(\LC_{\ell})]$ be a basis of $K$ and $m\coloneqq \ords$.

    As $\Psi$ is $C_S$-equivariant, the action of $\numg{\Psi}$ preserves $\Image\pi^{\ast}$ and $K=\ker\pi_{\ast}$.
    Therefore, it can be written as
    \begin{equation}
        \numg{\Psi} = 
        \begin{pmatrix}
            \numg{\Phi} & 0 \\
            0 & A
        \end{pmatrix}, 
    \end{equation}
    where $A$ is the invertible matrix associated with the above basis.

    The numerical Grothendieck group $N(X)$ (resp. $N(S)$) has the lattice structure as a sublattice of the Mukai lattice $H^{\ast}(X, \ZZ)$ (resp. $H^{\ast}(S, \ZZ)$).
    The Hodge index theorem induces that these lattices $N(X)$ and $N(S)$ have signature $(2, \rho_X)$ and $(2, \rho_S)$, respectively.
    Thus, $K$ is negative definite of signature $(0, \rho_X -\rho_S)$, and it follow that 
    $h\coloneqq -\langle-,-\rangle: K_{\RB}\times K_{\RB} \to \RB$ is an inner product, where $\langle-,-\rangle$ is the Mukai pairing.

    Then, since $\Psi$ is an equivalence, for any $v, w\in N(X)_{\RB}$, 
    \[
    h(w, v) = h(Aw, Av).
    \]
    Therefore, $A$ is unitary, and thus, we have $|\lambda| =1$ for any eigenvalue $\lambda$ of $A=\numg{\Psi}|_{K}$.
\end{proof}

\subsubsection{Ploog's Results}\label{subsubsection: ploog's results}
The rest of this subsection is devoted to recalling the results of \cite[Section 3.3]{ploog_2007_equivariant_autoequivalences_for_finite_group_actions} combined with \cite{bridgeland_maciocia_2017_fouriermukai_transforms_for_quotient_varieties} that will be used in \cref{subsection: case of enriques surface}.
Ploog  showed that there exists an exact sequence
\begin{equation}
    0 \to \Auteq(D^b(X))^{C_S} \to \eqAuteq(D^b(X)) \to \Aut C_S \to 0, 
    \label{equation: descent for aut of canonical cover}
\end{equation}
where 
\[
    \eqAuteq(D^b(X))\coloneqq \{ (F, \mu)\in\Auteq(D^b(X))\times \Aut(C_S) \mid g_{\ast}F\cong F\circ\mu(g)_{\ast}, {}^{\forall}g\in C_S\}
\]
and the second map is the projection.

Let 
$\For: D^b_{C_S}(X) \to D^b(X)$ be the forgetful functor and  
$\Inf: D^b(X)\to D^b_{C_S}(X)$
be the inflation functor defined by $\Inf(E) = \bigoplus_{g\in C_S} g^{\ast}E$.
Let 
\[
\Auteq_{\Delta}^{C_S}(D^b(X))\coloneqq \{(P, \rho)\in D^{G_{\Delta}}(X\times X) \mid \text{Fourier-Mukai functor } \Phi_P \text{ is an equivalence}\}.
\]
Then there exists the following commutative diagram:
\begin{equation}
\begin{tikzcd}
\Auteq(D^b(X))^{C_S} \arrow[d, hook] & \Auteq_{\Delta}^{C_S}(D^b(X)) \arrow[l, "\For", two heads] \arrow[r, "\Inf"] & \Auteq(D^b(S)) \arrow[d] \\
\eqAuteq(D^b(X)) \arrow[rr, two heads] & &\eqAuteq(D^b(X))/C_S.
\end{tikzcd}
\label{equation: commutative diagram}
\end{equation}

\subsection{Proof for bielliptic surface}\label{subsection: end of proof}
\begin{thm}\label{theorem: Gromov-Yomdin equality for bielliptic surface}
    Let $S$ be a bielliptic surface and $\Phi\in\Auteq D^b(S)$.
    Then, the following holds:
    \[
        \hcat(\Phi) = \log(\rho(\numg{\Phi})).
    \]
\end{thm}
\begin{proof}

Let $\Phi\in\Auteq D^b(S)$ and $\Psi$ be the $G$-equivariant lift to $D^b(X)$ given in \cref{theorem: equivalence induce eq of canonical cover}.

Combining \cref{theorem: Gromov-Yomdin for an abelian surface}, \cref{proposition: cat entropy of lift}, and \cref{lemma: eigenvalue of Psi on K}, we have
\[
    \hcat(\Phi) = \hcat(\Psi) = \log\rho(\numg{\Psi}) = \log\rho(\numg{\Phi}).
\qedhere
\]
\end{proof}

\subsection{Counterexample of Kikuta-Takahashi Conjecture on Enriques Surfaces}\label{subsection: case of enriques surface}
In this \cref{subsection: case of enriques surface}, we apply the discussions of \cref{subsection: canonical cover} to the case of an Enriques surface and its canonical covering.

Let $X$ be a projective K3 surface, $T_{\OO_X}$ the spherical twist associated to $\OO_X$, and $H$ an ample divisor on $X$ of digree $H^2\ge10$. 
In \cite{ouchi_2020_on_entropy_of_spherical_twists} Ouchi showed that 
\begin{equation}
\Psi_X\coloneqq T_{\OO_X}\circ (-\otimes \OO(-H))
\label{equation: counterexample of GY equality on K3}
\end{equation}
satisfies $\hcat(\Psi_X) > \log\rho(\numg{\Psi}_X)$. 
In particular, any K3 surface admits a counterexample to the conjectured Gromov-Yomdin type equality for categorical entropies.

For the remainder of this paper, we denote an Enriques surface by $S$ and its canonical cover by $\pi:X\to S$.
In this case, $C_S = \ZZ/2\ZZ$ and let $\iota$ be the generator of $C_S$.

\begin{prop}\label{proposition: Enriques surface esixtence of Phi}
There exists $\Phi_S\in\Auteq(D^b(S))$ whose $C_S$-equivariant lift is $\Psi_X$.
\end{prop}
\begin{proof}
    Let $H_S$ be an ample divisor on $S$. 
    As the canonical cover $\pi$ is finite, $H_X\coloneqq \pi^{\ast}H_S$ is ample again. 
    The functor $-\otimes\OO(-H_X)$ commutes with $\iota_{\ast}$
    since $\OO(-H_X)\cong\iota_{\ast}\OO(-H_X)$.
    
    As $\OO_X\cong\iota_{\ast}\OO_X$, 
    \[
        \iota_{\ast}\circ T_{\OO_X}\cong T_{\iota_{\ast}\OO_X}\circ \iota_{\ast} \cong T_{\OO_X}\circ \iota_{\ast}. 
    \]

    Thus, \eqref{equation: descent for aut of canonical cover} and \eqref{equation: commutative diagram} show the claim.
\end{proof}

\begin{thm}\label{theorem: counter eg Enriques}
    Let $\Phi_S\in\Auteq D^b(S)$ be an autoequivalence constructed in \cref{proposition: Enriques surface esixtence of Phi}.
    Then, 
    \[
    \hcat(\Phi_S) > \log\rho(\numg{\Phi}_S).
    \]
\end{thm}
\begin{proof}
    \cref{proposition: cat entropy of lift} and \cref{lemma: eigenvalue of Psi on K} show that Ouchi's counterexample on K3 surface descends to the Enriques surface.
\end{proof}

\bibliographystyle{amsalpha}
\bibliography{bibtex_tyoshida}
\end{document}